\UseRawInputEncoding
\documentclass{amsart}
\usepackage{amssymb}
\usepackage{amsmath,amssymb,hyperref,mathrsfs,color}
\usepackage{paralist}
\usepackage{amssymb}
\usepackage{amssymb}
\usepackage{amsmath}
\usepackage{enumitem}
\usepackage[]{hyperref}
\usepackage{fancyhdr}
\usepackage{calc}
 \newtheorem{The}{Theorem}[section]
 \newtheorem{Cor}[The]{Corollary}
 \newtheorem{Lem}[The]{Lemma}
 \newtheorem{Pro}[The]{Proposition}
 \theoremstyle{definition}
 \newtheorem{defn}[The]{Definition}
 \theoremstyle{remark}
 \newtheorem{Rem}[The]{Remark}
 \newtheorem{ex}{Example}
 \numberwithin{equation}{section}

\allowdisplaybreaks[4]
\begin{document}
\newcommand{\mytitle}{Metric viscosity solutions and distance-like functions on the Wasserstein space}
\newcommand{\myauthor}{Huajian Jiang \and Xiaojun Cui}
\title{Metric viscosity solutions and distance-like functions on the Wasserstein space}
\author{Huajian Jiang \and Xiaojun Cui}
\address{School of Mathematics, Nanjing University,
Nanjing 210093, China}
\email{huajianjiang@smail.nju.edu.cn}
\address{School of Mathematics, Nanjing University,
Nanjing 210093, China}
\email{xcui@nju.edu.cn}
\subjclass[2010]{58E10, 60B10, 49L25}
\date{\today}
\keywords{Wasserstein space, Eikonal equation, Distance-like function, Viscosity solution}
\thanks{Xiaojun Cui and Huajian Jiang are supported by the National Natural Science Foundation of China (Grant No. 12171234), the Project Funded by the Priority Academic Program Development of Jiangsu Higher Education Institutions (PAPD) and the Fundamental Research Funds for the Central Universities}
\maketitle

\begin{abstract}
Viscosity solutions to the eikonal equation $|\triangledown u|_g=1$, known to be exactly distance-like functions, on a non-compact complete Riemannian manifold $(\mathcal{M},g)$ are crucial for understanding the underlying geometric and topological properties. In this work, we explore metric viscosity solutions, distance-like functions and their relationship on a metric space, especially on the Wasserstein space $\mathcal{P}_{p}(\mathcal{X})$ where $\mathcal{X}$ is a complete, separable, locally compact and non-compact geodesic space. Meanwhile, we provide two distinct ways to construct (strong) metric viscosity solutions on $\mathcal{P}_{p}(\mathcal{X})$ and study their properties.
\end{abstract}

\tableofcontents

\newpage

\section{Introduction}
On a non-compact complete Riemannian manifold $(\mathcal{M},g)$, viscosity solutions to the eikonal equation
$|\triangledown u|_g=1$
are fundamental in studying the geometry, topology and dynamics of geodesic flows.
Distance-like function, abbreviated as dl-function, represents the distance to the infinity as defined in Definitions \ref{defn1} and \ref{defn2}. The two definitions are consistent on the Riemannian manifold and have been widely investigated in the literature, as exemplified by the research in \cite{HW}.
The 1-Lipschitz property of a dl-function guarantees it is a viscosity subsolution of the eikonal equation $|\triangledown u|_g=1$. The existence of a negative gradient curve (Definition \ref{NGC}) for a dl-function implies it acts as a viscosity supersolution of the eikonal equation.
In conclusion, a function is a dl-function if and only if it is a viscosity solution of the eikonal equation on such a manifold, as demonstrated in \cite{CX1,CX2,Liang}.

However, the classical theory of viscosity solution relies heavily on the differential structure, hindering its extension to general metric spaces. Fortunately, various kinds of metric viscosity solutions have been provided in \cite{LA2,YG,LQ,12,13,HQ} on metric spaces, and these metric viscosity solutions only depend on the metric structure. For their concrete definitions and relationships among them, refer to \cite{QZ}.
Let $(\mathcal{Y},d)$ be a metric space and $\Omega \subset \mathcal{Y}$ be an open set, the so-called eikonal equation in \cite{YG,LQ} is
\begin{equation*}
|\triangledown u|(y)=f(y) \ \text{in} \ \Omega.
\end{equation*}
Here $f:\Omega \rightarrow [0,\infty)$ is a given continuous function satisfying $\inf_{\Omega}f>0$ and $|\triangledown u|(y):=\limsup_{x \rightarrow y}\frac{|u(x)-u(y)|}{d(x,y)}$ stands for a generalized notion of the gradient norm of $u$ in metric spaces.
To highlight the connection between metric viscosity solutions and the metric, we only consider the eikonal equation of the simplest type as follows:
\begin{equation}\label{eikonal}
|\triangledown u|(y)=1 \ \text{in} \ \Omega.
\end{equation}

We first recall the concept of the local slope of a function.
In \cite{LA1}, let $(\mathcal{Y},d)$ be a length space (A length space is a metric space where for any given two points the distance between them is equal to the infimum of the lengths of all rectifiable curves joining them) and $\Omega \subset \mathcal{Y}$ be an open set, the local slope of a function $u: \mathcal{Y} \rightarrow \mathbb{R}$ is defined as
\begin{equation*}
|\partial u|(y):=\limsup_{x \in \Omega, x \rightarrow y} \frac{(u(y)-u(x))^{+}}{d(y,x)}.
\end{equation*}

\begin{defn}[Metric Viscosity Solution]
Let $(\mathcal{Y},d)$ be a length space and $\Omega \subset \mathcal{Y}$ be an open set. A locally Lipschitz function $u:\Omega \rightarrow \mathbb{R}$ is called a metric viscosity solution of \eqref{eikonal} if for any $y \in \Omega$
\begin{equation}
|\partial u|(y)=1.
\end{equation}
\end{defn}
For convenience, throughout this paper we use ``metric viscosity solution'' instead of ``metric viscosity solution of \eqref{eikonal}''. The metric viscosity solution is also called Monge solution of \eqref{eikonal} in \cite{LQ,AD,RT}.

In this work, we are mainly concerned with  the metric viscosity solutions and the dl-functions on the Wasserstein space. The Wasserstein space is of fundamental importance in optimal transport theory. Over the past three decades, optimal transport theory has flourished due to its deep connections with geometry, analysis, probability, and various other fields in mathematics (see, e.g., \cite{CV}). It has played a crucial role in areas such as geometric inference \cite{JD,FC}, generative adversarial networks \cite{MA2}, clustering analysis \cite{LM}, mean-field games \cite{RC}, machine learning \cite{MA1}, option pricing \cite{BJ}, stochastic portfolio theory \cite{EE}, and so on.
The Wasserstein space provides a rigorous framework for optimal transport theory.
For a Polish space $(\mathcal{X},d)$, the Wasserstein space $\mathcal{P}_{p}(\mathcal{X})$, comprising Borel probability measures on $\mathcal{X}$ with finite $p$-moment, is itself Polish \cite{CV}. Moreover, if $\mathcal{X}$ is a length (resp. geodesic) space, then $\mathcal{P}_{p}(\mathcal{X})$ is also a length (resp. geodesic) space \cite{SL}. The geometric richness of $\mathcal{P}_{p}(\mathcal{X})$ has found wide applications across diverse fields \cite{LA1,CV}.
Hence, a detailed investigation of the geometric structure of $\mathcal{P}_{p}(\mathcal{X})$ and its implications is highly important.

Notable progress has been made in the study of dl-functions, especially in the Wasserstein space. For instance, the following particular kinds of dl-functions should be mentioned. The Busemann functions on the Wasserstein space have been presented in \cite{ZG1}, and the horo-functions associated with atom probability measures have been explored in \cite{ZG2}. In the following parts, we will provide the definitions of two kind of dl-functions on metric spaces.

First, we adopt the definition of the dl-function in \cite{CX2} for an unbounded metric space, denoted as dl$_{C}$-function.
\begin{defn}[Dl$_{C}$-function]\label{defn1}
On an unbounded metric space $(\mathcal{Y},d)$, consider a sequence $\{H_{n}\}$ of non-empty closed subsets of $\mathcal{Y}$ such that, for a given (hence, for any) $y_{0} \in \mathcal{Y},$ the distance between $y_0$ and $H_n$ diverges to infinity as $n$ goes to infinity.
A function $u: \mathcal{Y} \rightarrow \mathbb{R}$ is called a dl$_{C}$-function if there exists a sequence $\{c_{n}\} \subset \mathbb{R}$ such that
$$u(y)=\lim_{n \rightarrow \infty}[d(y,H_{n})-c_{n}] \ \text{for all} \ y \in \mathcal{Y}.$$
\end{defn}

Another definition of the dl-function, introduced by Gromov \cite{MG1}, is named as dl$_{G}$-function.
\begin{defn}[Dl$_{G}$-function]\label{defn2}
On an unbounded metric space $(\mathcal{Y},d)$, a function $u: \mathcal{Y} \rightarrow \mathbb{R}$ is called a dl$_{G}$-function if for each $y \in \mathcal{Y}$,
\begin{equation*}
u(y)=c+d(y,u^{-1}(-\infty,c]) \ \text{for all} \ c<u(y).
\end{equation*}
\end{defn}

\begin{Rem}
Here, we adopt the convention that $d(y,H)=\infty$  if $H$ is an empty set. It is clear that the dl$_{C}$-function and dl$_{G}$-function are both  1-Lipschitz. On a complete, separable, locally compact and non-compact geodesic space, a function is a dl$_{C}$-function if and only if it is a dl$_{G}$-function. On a general unbounded metric space, dl$_{G}$-functions are necessarily dl$_{C}$-functions (see Corollary \ref{DLG}). Conversely, this is not always true, as illustrated in Example \ref{ex3}.
\end{Rem}

In the following theorem, we provide several equivalent characterizations of metric viscosity solutions on a complete unbounded length space, and thus reveal their relationship with the dl$_{G}$-functions.

\begin{The}\label{VS}
Let $(\mathcal{Y},d)$ be a complete unbounded length space and $u: \mathcal{Y} \rightarrow \mathbb{R}$ be a continuous function. Then the following several statements are equivalent:
\begin{itemize}
\item[(i)]{$u$ is a metric viscosity solution;}
\item[(ii)]{$u$ is a dl$_{G}$-function on $\mathcal{Y}$.}
\item[(iii)]{For any $y \in \mathcal{Y}$ and $r>0$,
\begin{equation*}
u(y)=\inf_{x \in \partial B_{r}(y)} \{u(x)+d(x,y)\}.
\end{equation*}
where $\partial B_{r}(y)$ is the sphere centered at $y$ with radius $r$.}
\end{itemize}
\end{The}

\begin{defn}[$\varepsilon$-Negative Gradient Curve]
Let $(\mathcal{Y},d)$ be a complete length space and $\Omega \subset \mathcal{Y}$ be an open set. Let $u$ be a 1-Lipschitz solution on $\Omega$. For any $\varepsilon>0$, an $\varepsilon$-negative gradient curve of $u$ is a regular curve (Definition \ref{RC}) $\gamma:[0,r] \rightarrow \Omega$ such that for all $0 \leq t_1 \leq t_2 \leq r$
\begin{equation*}
u(\gamma(t_{1}))-u(\gamma(t_{2})) \geq d(\gamma(t_{1}),\gamma(t_2))-\varepsilon.
\end{equation*}
\end{defn}
On a non-compact complete Riemannian manifold $(\mathcal{M},g)$, for any viscosity solution $u$ of the eikonal equation $|\triangledown u|_g=1$, for any point $x$ there exists at least one global negative gradient curve $\gamma$ starting from $x$ (i.e., $\gamma: [0,\infty) \rightarrow \mathcal{M}$ satisfies $\gamma(0)=x$ and $u(\gamma(t_1))-u(\gamma(t_2))=d(\gamma(\gamma_1),\gamma(t_2))$ for all $t_2 \geq t_1 \geq0$). Similarly, for the metric viscosity solutions, analogous conclusion for $\varepsilon$-negative gradient curves can be drawn as follows.

\begin{The}\label{CC}
Let $u$ be a metric viscosity solution on a complete unbounded length space $(\mathcal{Y},d)$. Then for any point $y_0 \in \mathcal{Y}$ and for every $\varepsilon > 0$, there exists a regular curve $\gamma:[0,\infty) \to \mathcal{Y}$ initiated from $y_0$ such that
\begin{equation*}
u(\gamma({t_1})) - u(\gamma({t_2})) \geq d(\gamma({t_2}),\gamma({t_1})) - \varepsilon
\end{equation*}
for all $t_2 \geq t_1 \geq 0$ and $\lim_{t \rightarrow \infty}d(y_0,\gamma(t))=\infty$.
\end{The}

Next, we investigate the existence of the metric viscosity solutions in the unbounded Wasserstein space $\mathcal{P}_{p}(\mathcal{X})$, where $\mathcal{X}$ is a complete, separable, locally compact and non-compact geodesic space. For the sake of clarity in the subsequent analysis, we present the following definition.

\begin{defn}[Strong Metric Viscosity Solution]
Let $(\mathcal{Y},d)$ be a complete length space and $\Omega \subset \mathcal{Y}$ be an open set. A metric viscosity solution $u: \Omega \rightarrow \mathbb{R}$ is called a strong metric viscosity solution if for each $y \in \Omega$, there exists at least one negative gradient curve starting from $y$, rather than just an $\varepsilon$-negative gradient curve.
\end{defn}

Despite lacking local compactness, $\mathcal{P}_{p}(\mathcal{X})$ still admits strong metric viscosity solutions. Examples include the Busemann functions \cite{ZG1} and the horo-functions related to atom probability measures \cite{ZG2}. Next, we present two ways to construct more general strong metric viscosity solutions on $\mathcal{P}_p(\mathcal{X})$.
\begin{itemize}
\item{We can construct strong metric viscosity solutions from a special class of dl$_{C}$-functions on $\mathcal{P}_{p}(\mathcal{X})$, elaborated in Theorem \ref{The1.2};}
\item{We are also able to construct strong metric viscosity solutions on $\mathcal{P}_{p}(\mathcal{X})$ from metric viscosity solutions on $\mathcal{X}$, as demonstrated in Theorem \ref{Pro111}.}
\end{itemize}

\begin{defn}[(CS) Condition]
Consider a complete unbounded geodesic space $(\mathcal{Y},d)$ (A geodesic space is a metric space in which every pair of points is joined by a curve whose length is equal to the distance between them), we say a sequence $\{y_{n}\} \subset \mathcal{Y}$ with $\lim_{n \rightarrow \infty } d(y_{n},y_{0})=\infty$ for some $y_{0} \in \mathcal{Y}$ satisfies (CS) condition if there exist $\sigma>0$ and a unit-speed geodesic $\mu^{n}$ from $y_{0}$ to $y_{n}$ for each $n$ such that
$$\partial B_{\sigma}(y_{0})\cap \{ \{\mu^{n}(t)\}_{0 \leq t \leq d(y_{0},y_{n})}   \}_{n \in \mathbb{N}} \ \text{has convergent subsequences}.$$
\end{defn}

\begin{Rem}
It is important to note that on $\mathcal{P}_{p}(\mathcal{X}) \ (p>1)$, whether the (CS) condition holds is independent of specific choices of the $\sigma$, $y_{0}$, and geodesics $\{\mu^{n}\}$, see Lemma \ref{Lem3.10}. The (CS) condition is satisfied in Examples \ref{Ex1} and \ref{ex111} but not in Example \ref{ex5} on $\mathcal{P}_{p}(\mathcal{X})$. However, if the space $\mathcal{Y}$ is locally compact, it is clear that the condition always holds.
\end{Rem}

\begin{defn}[(CS) Condition for Closed Sets]
On a complete unbounded geodesic space $(\mathcal{Y},d)$, we say a sequence $\{H_n\}$ of closed subsets in the space satisfies (CS) condition if, for any $y \in \mathcal{Y}$, there exists a sequence $\{y^{n}\}$ with $y^{n} \in H_{n}$ for all $n \in \mathbb{N}$ such that
\begin{equation*}
\lim_{n \rightarrow \infty}[d(y,H_{n})-d(y,y^{n})]=0,
\end{equation*}
and the sequence $\{y^{n}\}$ itself satisfies the (CS) condition.
\end{defn}

By Example \ref{ex3}, a dl$_{C}$-function may not be a metric viscosity solution on $\mathcal{P}_{p}(\mathcal{X})$. But a special kind of dl$_{C}$-function, which is generated by a sequence of closed sets satisfying the (CS) condition, is even a strong metric viscosity solution, as described below.

\begin{The}\label{The1.2}
On the Wasserstein space $\mathcal{P}_{p}(\mathcal{X}) \ (p>1)$, the dl$_{C}$-function induced by a sequence of closed sets satisfying the (CS) condition must be a strong metric viscosity solution.
\end{The}

Now we would construct strong metric viscosity solutions from another angle. The following result show that any metric viscosity solution on the ambient space can induce a solution on the Wasserstein space, which provides a distinct proof for the existence of the strong metric viscosity solutions. The fact also indicates that the strong metric viscosity solutions on the Wasserstein space are abundant.

\begin{The}\label{Pro111}
Let $u$ be a (strong) metric viscosity solution on a complete, separable, locally compact and non-compact geodesic space $\mathcal{X}$, then the function $\hat{u}$ defined by $\hat{u}(\omega)=\int_{\mathcal{X}}u(x)d\omega$ is still a strong metric viscosity solution on $\mathcal{P}_{p}(\mathcal{X})$ ($p \geq 1$).
\end{The}

\begin{Rem}
In this study, we are merely concerned with the existence of the strong metric viscosity solutions on the geodesic space $\mathcal{P}_{p}(\mathcal{X})$. Suppose that $(\mathcal{Y},d)$ is not a geodesic space but only a complete length space. It is straightforward to observe that for any $y \in \mathcal{Y}$, $d(\cdot, y)$ is a metric viscosity solution on $\mathcal{Y} \setminus \{y\}$, but there exists $y_0 \in \mathcal{Y}$ such that  $d(\cdot,y_0)$ is not a strong metric viscosity solution. However, on a geodesic space, there are no specific examples of non-strong metric viscosity solutions, although we believe that such examples should exist in abundance.
\end{Rem}

The organization of the remainder of this paper is as follows: Section 2 presents basic theories and definitions related to metric spaces and Wasserstein spaces as background knowledge. In Section 3, we prove Theorems \ref{VS}, \ref{CC} and give an important example, and thus reveal the relationship between the metric viscosity solutions and the dl-functions. To prove Theorem \ref{The1.2}, a key lemma \ref{The1.1} for the existence of co-rays in $\mathcal{P}_p(\mathcal{X})$ is first established in Section 4. Subsequently, in this section, we prove Theorems \ref{The1.2} and \ref{Pro111}, and study the properties of the strong metric viscosity solutions.

\section{Preliminaries}
In this section we review some concepts of metric spaces and Wasserstein spaces, mentioned in the first section.

Let $(\mathcal{Y},d)$ be a metric space. The length of a continuous curve $\gamma: [a,b] \rightarrow \mathcal{Y}$ is defined by
$$L(\gamma)=\sup_{N \in \mathbb{N}} \sup_{a=t_{0}<t_{1}<\cdots<t_{N}=b}\sum_{i=0}^{N-1}d(\gamma(t_{i}),\gamma(t_{i+1})).$$
Note that if $L(\gamma)< \infty$, then the real-valued function $s_\gamma: [a,b] \rightarrow [0,L(\gamma)]$ defined by
\begin{equation*}
s_{\gamma}(t):=L(\gamma \mid_{[a,t]})
\end{equation*}
is a monotone increasing function, and hence it is differentiable at almost every $t \in [a,b]$. We denote $s^{'}_{\gamma}$ by $|\gamma^{'}|$, which can be equivalently defined by
\begin{equation*}
|\gamma^{'}|(t)=\lim_{\tau \rightarrow 0} \frac{d(\gamma(t+\tau),\gamma(t))}{|\tau|}
\end{equation*}
for almost every $t \in (a,b)$. We say $|\gamma^{'}|$ is the speed of the curve $\gamma$ on $[a,b]$.

Let $I \subset \mathbb{R}$ be an interval. We say a curve $\gamma: I \rightarrow \mathcal{Y}$ is rectifiable if $L(\gamma \mid_{[a,b]})< \infty$ for each bounded closed interval $[a,b] \subset I$.
\begin{defn}[Regular Curve]\label{RC}
Let $(\mathcal{Y},d)$ be a metric space and $I \subset \mathbb{R}$ be a non-degenerate interval. A rectifiable curve $\gamma: I \rightarrow \mathcal{Y}$ with $|\gamma^{'}|(t) \neq 0$ for almost every $t \in I$ is called a regular curve.
\end{defn}

A metric space $(\mathcal{Y},d)$ is said to be a length space if for any $x,y \in \mathcal{Y}$
\begin{equation}\label{2.2}
d(x,y)=\inf_{\gamma \in C([0,1],\mathcal{Y})}\{L(\gamma)\mid \gamma({0})=x,\gamma({1})=y\}.
\end{equation}
$\mathcal{Y}$ is a geodesic space if the infimum in \eqref{2.2} is attainable for any $x,y \in \mathcal{Y}$. A continuous curve $\gamma$ is called a constant-speed minimizing geodesic segment if $$d(\gamma({s}),\gamma({t}))=\frac{|t-s|}{b-a}d(\gamma({a}),\gamma({b}))\ \text{for any} \ s,t \in [a,b].$$
For convenience, throughout this paper we use the single word ``geodesic" instead.

\begin{defn}[Ray]\label{ray}
Let $(\mathcal{Y},d)$ be an unbounded metric space. A ray is a regular curve $\gamma \in C(\mathbb{R}_{+}, \mathcal{Y})$ such that $d(\gamma({s}),\gamma({t}))=|s-t|d(\gamma({0}),\gamma({1}))$ for all $s,t \geq 0$ where $k_{\gamma}:=d(\gamma(0),\gamma(1))$ is called the speed of $\gamma.$
\end{defn}

\begin{defn}[Negative Gradient Curve]\label{NGC}
Let $u$ be a 1-Lipschitz function on a metric space $(\mathcal{Y},d)$. A regular curve $\gamma: I \rightarrow \mathcal{Y}$ is called a negative gradient curve of $u$ if for any $[t_1,t_2] \subset I$
\begin{equation*}
u(\gamma({t_1}))-u(\gamma({t_2}))=d(\gamma(t_2),\gamma(t_1)).
\end{equation*}
In particular, if $\gamma$ is a ray, it is then called a negative gradient ray of $u$.
\end{defn}

Consider a Polish space $(\mathcal{Y},d)$. The associated Wasserstein space $\mathcal{P}_{p}(\mathcal{Y})$ of order $p \geq 1$ comprises Borel probability measures on $\mathcal{Y}$ with finite $p$-moment, defined as:
$$\mathcal{P}_{p}(\mathcal{Y})=\left\{\mu \in \mathcal{P}(\mathcal{Y})\left| \int_{\mathcal{Y}}d^{p}(y_{0},y)d\mu(y)<\infty \right.\right\},$$
where $y_{0} \in \mathcal{Y}$ is arbitrarily fixed. The $p$-Wasserstein metric is given by
$$W_{p}(\mu,\nu)=\left(\inf_{\pi \in \Pi(\mu,\nu)}\int_{\mathcal{Y} \times \mathcal{Y}}d^{p}(x,y) \pi(dx,dy)\right)^{\frac{1}{p}},$$
where $\Pi(\mu,\nu)$ is the set of all couplings between $\mu$ and $\nu$. The set of optimal couplings is denoted as $Opt(\mu,\nu)$. For further details, see \cite{LA1,LA3,CV,AF}.

\begin{defn}[Push-forward]
Let $f: \mathcal{X}_{1} \rightarrow \mathcal{X}_{2}$ be a Borel measurable map between Polish spaces, and $\mu$ be a Borel measure on $\mathcal{X}_{1}$. The push-forward of $\mu$ under $f$, denoted by $f_{\sharp}\mu$, is defined by $(f_{\sharp}\mu)(A)=\mu(f^{-1}(A))$ for any Borel subset $A$ of $\mathcal{X}_{2}.$
\end{defn}

The following statement is twofold: The Wasserstein space over a complete separable locally compact length space is a geodesic space; Geodesics in such a Wasserstein space can be considered as probability measures concentrated on the set of geodesics in the ambient space.

\begin{Lem}[{\cite[Proposition 2.5]{ZG1}}]\label{Lem2.4}
Let $p > 1$ and $(\mathcal{Y},d)$ be a complete separable, locally compact length space. Given $\mu,\nu \in \mathcal{P}_{p}(\mathcal{Y})$, let $L=W_{p}(\mu,\nu)$. Then for any continuous curve $(\mu(t))_{0 \leq t \leq L}$ in $\mathcal{P}_{p}(\mathcal{Y})$ with $\mu(0)=\mu,\ \mu(L)=\nu,$ the following properties are equivalent:
  \begin{enumerate}
    \item[(i)]{$\mu(t)$ is the law of $\gamma(t)$, where $\gamma: [0,L] \rightarrow \mathcal{Y}$ is a random geodesic such that $(\gamma(0),\gamma(L))$ is an optimal coupling.}
    \item[(ii)]{$(\mu(t))_{0 \leq t \leq L}$ is a unit-speed geodesic in the space $\mathcal{P}_{p}(\mathcal{Y})$.}
  \end{enumerate}
Moreover, for any given $\mu,\nu \in \mathcal{P}_{p}(\mathcal{Y})$ there exists at least one such curve.
\end{Lem}

Let $\mathcal{Y}$ be a metric space. We say that a sequence $\{\mu^{n}\} \subset \mathcal{P}(\mathcal{Y})$ converges weakly to $\mu \in \mathcal{P}(\mathcal{Y})$, denoted by $\mu^{n} \Rightarrow \mu$, if
\begin{equation*}
\lim_{n \rightarrow \infty}\int_{\mathcal{Y}}f(y)d\mu_{n}(y)=\int_{\mathcal{Y}}f(y)d\mu(y)
\end{equation*}
for every bounded continuous function $f$ on $\mathcal{Y}$. A subset $S \subset \mathcal{P}(\mathcal{Y})$ is tight if, for every $\varepsilon>0$, there exists a compact set $K_{\varepsilon} \subset \mathcal{Y}$ with $\mu(K_{\varepsilon})>1-\varepsilon$ for all $\mu \in S.$ The following theorem indicates why the tightness makes sense.

\begin{The}[Prokhorov, {\cite[Theorem 5.1.3]{LA1}}]\label{The2.6}
Let $\mathcal{Y}$ be a Polish space. A subset $S \subset \mathcal{P}(\mathcal{Y})$ is tight if and only if it is relatively compact in $\mathcal{P}(\mathcal{Y})$.
\end{The}

If $\mathcal{Y}$ is a complete locally compact length space, then every closed ball in $\mathcal{Y}$ is compact (see \cite[Proposition 2.5.22]{Geo}). Hence the following lemma holds.
\begin{Lem}\label{CON}
Let $(\mathcal{Y},d)$ be a complete, separable, locally compact length space and $\{\mu_n\}$ be a bounded sequence in the Wasserstein space $\mathcal{P}_p(\mathcal{Y})$, then there exists a subsequence that converges weakly in $\mathcal{P}_p(\mathcal{Y})$.
\end{Lem}

\begin{proof}
Let $y \in \mathcal{Y}$, there exists a constant $C>0$ such that $W_p(\delta_y,\mu^n) \leq C$ for all $n \in \mathbb{N}$. Let $\pi^n$ be an optimal coupling between $\delta_y$ and $\mu^n$, we have
\begin{equation*}
\mu^n\{x \mid d(y,x) \geq R\}=\pi^n \{(y,x) \mid d(y,x) \geq R\} \leq \left[ \frac{W_p(\delta_y,\mu^n)}{R}\right]^p \leq (\frac{C}{R})^p.
\end{equation*}
Then $\{\mu^n\}$ is tight, by Theorem \ref{The2.6}, there exists a subsequence $\{\mu^{n_k}\}$ such that $\mu^{n_k} \Rightarrow \mu \in \mathcal{P}(\mathcal{Y})$. \cite[Lemma 2.1.12]{AF} implies $\mu \in \mathcal{P}_p(\mathcal{Y})$, we complete the proof.
\end{proof}

The next theorem provides an effective description of the convergence with respect  to the Wasserstein distance.
\begin{The}[{\cite[Theorem 7.1.5]{LA1}}]\label{The2.7}
Let $(\mathcal{Y},d)$ be a Polish space and $p \geq 1$. Let $\{\mu^{n}\}_{n \in \mathbb{N}} \subset \mathcal{P}_{p}(\mathcal{Y})$ and $\mu \in \mathcal{P}_{p}(\mathcal{Y})$, then the following statements are equivalent:
  \begin{enumerate}
    \item[(i)]{For some $y_0 \in \mathcal{Y}$, $\mu^{n} \Rightarrow \mu$ and
              $$\lim_{R \rightarrow \infty}\limsup_{n \rightarrow \infty}\int_{d(y_0,y) \geq R}d^p(y_0,y)d\mu^{n}(y)=0;$$}
    \item[(ii)]{$W_{p}(\mu^{n},\mu) \rightarrow 0$ as $n \rightarrow \infty$}.
  \end{enumerate}
\end{The}

The following Ascoli-Arzel$\grave{\text{a}}$ Theorem is a straight corollary from \cite[Theorem 5]{RSC} and \cite[Theorem 11.28]{ZS}.
\begin{The}\label{A-A}
Let $\mathcal{Y}$ be a separable metric space, $\mathcal{K}$ be a complete metric space and $\mathcal{H}$ be a collection of continuous maps from $\mathcal{Y}$ to $\mathcal{K}$.
If $\mathcal{H}$ is equi-continuous and for every $y \in \mathcal{Y}$, $H(y)=\left\{f(y): f \in \mathcal{H} \right\}$ is relatively compact in $\mathcal{K}$.
Then every sequence $\{f_n\}$ in $\mathcal{H}$ has then a subsequence that converges uniformly on every compact subset of $\mathcal{Y}$.
\end{The}

\section{The relationship between metric viscosity solutions and distance-like functions}
In this section, we will prove Theorems \ref{VS}, \ref{CC} and give an example that a dl$_C$-function is not a metric viscosity solution. Firstly, we introduce several auxiliary lammas.

\begin{Lem}\label{SSSSS}
Let $(\mathcal{Y},d)$ be a complete length space and $\Omega \subset \mathcal{Y}$ be an open set. Let $u$ be a 1-Lipschitz function on $\Omega$. If for any $y_{0} \in \Omega$, there exists $r_{0}>0$ such that
\begin{equation*}
u(y_{0})=\inf_{y \in \partial B_{r_{0}}(y_{0})}\{u(y)+d(y,y_{0})\}.
\end{equation*}
Then for any $0<r<r_{0}$, we have
\begin{equation*}
u(y_{0})=\inf_{y \in \partial B_{r}(y_{0})}\{u(y)+d(y,y_{0})\}.
\end{equation*}
\end{Lem}

\begin{proof}
Suppose the contrary, there exists $\delta>0$ such that for all $y \in \partial B_{r}(y_{0})$, $u(y_{0}) \leq u(y)+d(y,y_{0})-\delta$. Consider any point $y \in \partial B_{r_{0}}(y_{0})$, and let $\gamma$ be a continuous curve connecting $y_{0}$ to $y$ such that $L(\gamma) \leq d(y_0,y)+\frac{\delta}{2}$. Let $y^{'} \in \gamma \cap \partial B_{r}(y_{0})$, then
\begin{equation*}
\begin{aligned}
u(y_{0})&=u(y_{0})-u(y^{'})+u(y^{'})-u(y)+u(y)\\
&\leq [d(y^{'},y_{0})-\delta]+d(y^{'},y)+u(y)\\
&\leq L(\gamma)+u(y)-\delta\\
&\leq d(y_{0},y)+u(y)-\frac{\delta}{2}.
\end{aligned}
\end{equation*}
By the arbitrariness of $y \in \partial B_{r_{0}}(y_{0})$, we have $\inf_{y \in \partial B_{r_0}(y_0)}\{u(y)+d(y,y_0)\} \geq u(y_0)+\frac{\delta}{2}$ and this contradicts the assumption.
\end{proof}

In order to prove the crucial Lemma \ref{aaaa}, we first recall the minimal element version of the Zorn's Lemma.

\begin{Lem}[{\cite[Zorn's Lemma]{Zorn}}]
Let $\mathcal{S}$ be a non-empty partially ordered set in which every totally ordered subset has a lower bound. Then $\mathcal{S}$ contains at least one minimal element.
\end{Lem}

Next, in order to use the Zorn's lemma, we first define a suitable partial order relation.

\begin{defn}
Let $u$ be a function on a metric space $(\mathcal{Y},d)$ and $\delta>0$. We define a partial order relation $\succeq_\delta$ on $\mathcal{Y}$ as follows: for any $x,y \in \mathcal{Y}$, $x \succeq_{\delta} y$ if $u(x)-u(y) \geq \delta d(x,y)$.
\end{defn}

We verify that $\succeq_{\delta}$ indeed forms a partial order relation. It is trivial that $y \succeq_{\delta} y$ for all $y \in \mathcal{Y}$. Hence, we only need to check transitivity and antisymmetry.
\begin{itemize}
\item{Assume that $y_{1} \succeq_{\delta} y_{2}$ and $y_{2} \succeq_{\delta} y_{3}$ on $\mathcal{Y}$, it follows that $u(y_{1})-u(y_{2}) \geq \delta d(y_{1},y_{2})$ and $u(y_{2})-u(y_{3}) \geq \delta d(y_{2},y_{3})$. By the triangle inequality, we have $u(y_{1})-u(y_{3}) \geq \delta d(y_{1},y_{3})$, and thus $y_{1} \succeq_{\delta} y_{3}$.}
\item{Assume that $y_{1}\succeq_{\delta} y_{2}$ and $y_{2} \succeq_{\delta} y_{1}$ on $\mathcal{Y}$, then we have $-\delta d(y_{1},y_{2}) \geq u(y_{1})-u(y_{2}) \geq \delta d(y_{1},y_{2})$. This implies that $d(y_{1},y_{2})=0$ and thus $y_{1}=y_{2}$.}
\end{itemize}

Let $u$ be a 1-Lipschitz function on a complete unbounded metric space $(\mathcal{Y},d)$. For any $0<\delta<1$, $y_{0} \in \mathcal{Y}$ and $c<u(y_{0})$, denote
\begin{equation*}
\mathcal{U}_{c}(y_{0}) =\{ y \in \mathcal{Y} \vert u(y)>c \ \text{and} \ y_{0} \succeq_{\delta}y \}.
\end{equation*}
\begin{Lem}\label{aaaa}
Let $u$ be a 1-Lipschitz function on a complete unbounded length space $(\mathcal{Y},d)$. If for any $y \in \mathcal{Y}$, there exists $r>0$ such that
\begin{equation}\label{kkk}
u(y)=\inf_{x \in \partial B_{r}(y)}\{u(x)+d(x,y)\}.
\end{equation}
Then for any $y_0 \in \mathcal{Y}$ and $c<u(y_0)$, there exists a monotonically decreasing (with respect to order $\succeq_{\delta}$) chain $\{y_{k}\}_{k=0}^{\infty}$ in $\mathcal{U}_{c}(y_{0})$ starting from $y_{0}$ and the corresponding function values $\{u(y_{k})\} \rightarrow c.$
\end{Lem}

\begin{proof}
Assume the contrary that for every monotonically decreasing chain $\{y_k\}_{k=0}^{\infty}$ in $\mathcal{U}_{c}(y_{0})$ starting from $y_{0}$, the corresponding function values $\{u(y_{k})\} \nrightarrow c$.

Note that $\mathcal{U}_{c}(y_{0})$ constitutes a non-empty partially ordered set under the partial order $\succeq_{\delta}$. For any monotonically decreasing chain $\{y_{k}\}_{k=0}^{\infty} \subset \mathcal{U}_{c}(y_{0})$, the sequence $\{u(y_{k})\}$ is monotonically decreasing and has a lower bound. Hence, $\{u(y_{k})\}$ must be a Cauchy sequence. Moreover, from $u(y_{k-1})-u(y_{k}) \geq \delta d(y_{k-1},y_{k})$, we deduce that $\{y_{k}\}$ must also be a Cauchy sequence. By the completeness of $\mathcal{Y}$, there exists $y_{l} \in \mathcal{Y}$ such that $\lim_{k \rightarrow \infty}y_{k}= y_{l}$. By means of the continuity of $u$ and the given assumption, it is easy to verify that $y_{l}$ serves as a lower bound of the chain $\{y_{k}\}$ in $\mathcal{U}_{c}(y_{0})$.

Therefore, it is easily seen that every chain in $\mathcal{U}_{c}(y_{0})$ has a lower bound. By the Zorn's Lemma, there exists a minimal element $y_{*}$ in $\mathcal{U}_{c}(y_{0})$ such that for any element $y \in \mathcal{U}_{c}(y_{0})$, if $y_{*} \succeq_{\delta} y$, then $y=y_{*}$.

According to equality \eqref{kkk} and Lemma \ref{SSSSS}, there exists $y^{*} \in \partial B_{r^{'}}(y_{*})$ where $0<r^{'}<u(y_{*})-c$ such that
\begin{equation*}
u(y_{*})-u(y^{*}) \geq \delta d(y_{*},y^{*}).
\end{equation*}
By the 1-Lipschitz property of $u$, $u(y^{*}) \geq u(y_{*})-r^{'}>c$.
Thus, $y^{*} \in \mathcal{U}_{c}(y_{0}), \ y_{*} \succeq_{\delta} y^{*}$ and $y_{*} \neq y^{*}$. This contradicts the fact that $y_{*}$ is a minimal element in $\mathcal{U}_{c}(y_{0})$, thus we complete the proof.
\end{proof}

Let $u$ be a 1-Lipschitz function on a metric space $\mathcal{Y}$, for $c \in \mathbb{R}$, denote
\begin{equation*}
S_{c}=u^{-1}(-\infty,c]=\{y \in \mathcal{Y} \vert u(y) \leq c\}
\end{equation*}
and
\begin{equation*}
T_{c}=u^{-1}(c)=\{y \in \mathcal{Y} \vert u(y) = c\}
\end{equation*}

\begin{Lem}\label{WWWW}
Let $u$ be a 1-Lipschitz function on a complete unbounded length space $(\mathcal{Y},d)$. Suppose that, for any $y_{0} \in \mathcal{Y}$, there exists $r>0$ such that
\begin{equation*}
u(y_{0})=\inf_{y \in \partial B_{r}(y_{0})}\{u(y)+d(y,y_{0})\}.
\end{equation*}
Then for each $y_{0} \in \mathcal{Y},$ $u(y_{0})=c+d(y_{0},S_{c})$ for any $c<u(y_{0}).$ In other words, u is a dl$_{G}$-function.
\end{Lem}

\begin{proof}
For any $z \in S_{c}$, $u(y_{0})-c \leq u(y_{0})-u(z) \leq d(y_{0},z)$ follows from the 1-Lipschitz property of $u$. By the arbitrariness of $z$,
\begin{equation*}
u(y_{0})-c \leq d(y_{0},S_{c}) \ \text{for all} \ c<u(y_{0}).
\end{equation*}
Since $d(y_{0},T_{c}) \geq d(y_{0},S_{c})$, it suffices to show that
\begin{equation*}
u(y_{0}) \geq d(y_{0},T_{c})+c \ \text{for all} \ c<u(y_{0}).
\end{equation*}

Suppose the contrary that there exist $c<u(y_{0})$ and $\sigma>0$ such that $u(y_{0})-c \leq d(y_{0},T_{c})-\sigma$. This implies
\begin{equation*}
\frac{u(y_{0})-c}{d(y_{0},T_{c})} \leq  1-\frac{\sigma}{d(y_{0},T_{c})}.
\end{equation*}
In the following we will prove that there exists $y^{c} \in T_{c}$ such that
\begin{equation*}
\frac{u(y_{0})-c}{d(y_{0},y^{c})} \geq  1-\frac{\sigma}{2d(y_{0},T_{c})},
\end{equation*}
and thus get a contradiction.

By Lemma  \ref{aaaa}, we know that there exists a sequence $\{y_{k}\} \subset \mathcal{Y}$ such that
\begin{equation*}
y_{k-1} \succeq_{\delta} y_{k} \ \text{and} \ \lim_{k \rightarrow \infty} u(y_{k}) =c,
\end{equation*}
where $\delta=1-\frac{\sigma}{2d(y_{0},T_{c})}>0$.

As demonstrated in Lemma \ref{aaaa}, $\{y_{k}\}$ must be a Cauchy sequence. By the completeness of $\mathcal{Y}$, there exists $y^{c} \in \mathcal{Y}$ such that $\lim_{k \rightarrow \infty}y_{k}= y^{c}$. The continuity of $u$ guarantees
\begin{equation*}
y^{c} \in T_{c} \ \text{and} \ y_{0}\succeq_{\delta}y^{c}.
\end{equation*}
This indicates that
\begin{equation*}
y^{c} \in T_{c} \ \text{and} \ \frac{u(y_{0})-c}{d(y_{0},y^{c})} \geq  1-\frac{\sigma}{2d(y_{0},T_{c})}.
\end{equation*}
As a result, the lemma holds.
\end{proof}

\begin{Cor}\label{SSSS}
Let $u$ be a 1-Lipschitz function on a complete unbounded length space $(\mathcal{Y},d)$. If for any $y_{0} \in \mathcal{Y}$, there exists $r_{0}>0$ such that
\begin{equation*}
u(y_{0})=\inf_{y \in \partial B_{r_{0}}(y_{0})}\{u(y)+d(y,y_{0})\}.
\end{equation*}
Then for any $r>0$, we have
\begin{equation*}
u(y_{0})=\inf_{y \in \partial B_{r}(y_{0})}\{u(y)+d(y,y_{0})\}.
\end{equation*}
\end{Cor}

\begin{proof}
By Lemma \ref{SSSSS}, it suffices to analyze the case where $r_{0}<r$. Consider the set $S_{*}=\{y \in \mathcal{Y} \vert u(y_{0})-u(y) \geq 2r\},$ according to Lemma \ref{WWWW}, we have $u(y_{0})=d(y_{0},S_{*})+u(y_{0})-2r.$ It follows that $d(y,y_{0}) \geq 2r$ for all $y \in S_{*}$. If $u(y_{0})=\inf_{y \in \partial B_{r}(y_{0})}\{u(y)+d(y,y_{0})\}$ does not hold, a similar contradiction can be obtained as in the proof of Lemma \ref{SSSSS}. Hence, we complete the proof.
\end{proof}

The next Lemma is a straight corollary to \cite[Theorem 1.2]{LQ} and \cite[Theorem 5.5]{QZ}.
\begin{Lem}\label{local}
Let $(\mathcal{Y},d)$ be a complete length space and $\Omega \subset \mathcal{Y}$ be a open set. A function $u:\Omega \rightarrow \mathbb{R}$ is a metric viscosity solution if and only if:
\begin{itemize}
\item[(i)]{$u$ is 1-Lipschitz;}
\item[(ii)]{For any $y \in \Omega$, there exists $r>0$ such that
\begin{equation}\label{viscosity}
u(y)=\inf_{x \in \partial B_{r}(y)} \{u(x)+d(x,y)\}.
\end{equation}}
\end{itemize}
\end{Lem}

\begin{proof}[Proof of Theorem \ref{VS}]
By Corollary \ref{SSSS} and Lemma \ref{local}, $(i) \Leftrightarrow (iii),$ and thus we only need to prove that $(i)\Leftrightarrow(ii)$.

$(i) \Rightarrow (ii)$. The conclusion follows from Lemmas \ref{WWWW} and \ref{local}.

$(ii) \Rightarrow (i)$. For any $y_{0} \in \Omega$ and $c<u(y_{0})$, let $0<r<u(y_{0})-c$. Then there exist $y_{n} \in S_{c}=u^{-1}(-\infty,c]$ (where $y_{n} \notin B_{r}(y_{0})$ due to the 1-Lipschitz property of $u$) and a continuous curve  $\gamma_{n}$ connecting $y_{0}$ to $y_{n}$ such that $u(y_{0}) \geq d(y_{n},y_{0})+c-\frac{1}{2n}$ and $L(\gamma_n) \leq d(y_0,y_n)+\frac{1}{2n}$. We have
\begin{equation*}
\begin{aligned}
u(y_{0}) & \geq d(y_{n},y_{0})+c-\frac{1}{2n}\\
& \geq L(\gamma_n)+c-\frac{1}{n}\\
& \geq  d(y_{0},y^{n})+d(y^{n},y_{n})+u(y^{n})+c-u(y^{n})-\frac{1}{n}\\
& \geq d(y_{0},y^{n})+u(y^{n})-\frac{1}{n},
\end{aligned}
\end{equation*}
where $y^{n} \in \partial B_{r}(y_{0})\cap \gamma_{n}$. As $n$ goes to infinity, we have $u(y_0) \geq \inf_{x \in \partial B_r(y_0)}\{u(x)+d(x,y_0)\}$, and thus \eqref{viscosity} holds due to the 1-Lipschitz property of $u$. By Lemma \ref{local}, we obtain the conclusion.
\end{proof}

\begin{Cor}\label{DLG}
Let $(\mathcal{Y},d)$ be a complete unbounded length space. Then any metric viscosity solution and dl$_{G}$-function must be a $dl_{C}$-function.
\end{Cor}

\begin{proof}
By Theorem \ref{VS}, we only need to prove that any dl$_{G}$-function $u$ must be a $dl_{C}$-function. Let $H_n=u^{-1}(-\infty,-n]$, for any $y_0 \in \mathcal{Y}$, by the 1-Lipschitz property of $u$, we have
\begin{equation*}
d(y_0,H_n) \geq u(y_0)+n \rightarrow \infty \ \text{as} \ n \rightarrow \infty.
\end{equation*}
For each $y \in \mathcal{Y}$
\begin{equation}\label{DL}
u(y)=\lim_{n \rightarrow \infty}[d(y,H_n)-n].
\end{equation}
This completes the proof of the corollary.

\end{proof}

If $\mathcal{Y}$ is locally compact, the infimum in equality \eqref{viscosity} can be attained. This indicates that for any $y \in \mathcal{Y}$ there exists at least one negative gradient curve of $u$ initiated from $y$, also known as a calibrated curve.
Thus, the following corollary holds.

\begin{Cor}\label{The1.7}
Let $(\mathcal{M},g)$ be a complete Riemannian manifold and $\Omega \subset \mathcal{M}$ be an open set, a function $u: \Omega \rightarrow \mathbb{R}$ is a metric viscosity solution if and only if it is a classical viscosity solution of \eqref{eikonal}.
\end{Cor}

\begin{Rem}
From Corollary \ref{The1.7}, generally speaking, the concept of the metric viscosity solution in our study differs from the definition of the curve-based solution of \eqref{eikonal} in \cite{YG}, where a classical viscosity solution (classical solution) of \eqref{eikonal} may not necessarily be a curve-based solution of \eqref{eikonal}, see \cite[Example 2.4]{YG}.
\end{Rem}

\begin{proof}[Proof of Theorem \ref{CC}]
By Corollary \ref{SSSS}, it follows that
\begin{equation*}
u(y_0) = \inf_{y \in \partial B_1(y_0)} \{ u(y) + 1 \}.
\end{equation*}
This means that for any $\varepsilon>0$, there exists $y_1 \in \partial B_1(y_0)$ such that
\begin{equation}\label{calibrated}
u(y_0) - u(y_1) \geq d(y_0,y_1) - \frac{\varepsilon}{8}.
\end{equation}
Consider a continuous curve $\gamma^1:[0,1] \to \mathcal{Y}$ connecting $y_0$ to $y_1$ such that $L(\gamma^{1}) \leq d(y_0,y_1)+\frac{\varepsilon}{8}$. We claim that for all $t \in [0,1]$,
\begin{equation*}
u(y_0) - u(\gamma^1(t)) \geq d(y_0,\gamma^1(t)) - \frac{\varepsilon}{4}.
\end{equation*}
For a contradiction, let us assume that there exists $t_0 \in (0,1)$ such that
\begin{equation*}
u(y_0) - u(\gamma^1({t_0})) < d(y_0,\gamma^1(t_0)) - \frac{\varepsilon}{4}.
\end{equation*}
Then we have
\begin{equation*}
\begin{aligned}
u(y_0) - u(y_1) &= u(y_0) - u(\gamma^1({t_0})) + u(\gamma^1({t_0})) - u(y_1) \\
&< d(y_0,\gamma^1(t_0))- \frac{\varepsilon}{4} + d(\gamma^1({t_0}), y_1) \\
&\leq L(\gamma^1)-\frac{\varepsilon}{4}\\
& \leq d(y_0,y_1)-\frac{\varepsilon}{8}.
\end{aligned}
\end{equation*}
This contradicts inequality \eqref{calibrated}, and thus the claim holds.

For any $0 \leq t_1 \leq t_2 \leq 1$, using the 1-Lipschitz property of $u$, we have
\begin{equation}\label{ww}
\begin{aligned}
u(\gamma^1({t_1})) - u(\gamma^1({t_2})) &= (u(y_0) - u(\gamma^1({t_2}))) - (u(y_0) - u(\gamma^1({t_1}))) \\
&\geq d(y_0, \gamma^1({t_2})) - \frac{\varepsilon}{4} - d(y_0,\gamma^1({t_1})) \\
& \geq L(\gamma^1\mid_{[0,t_2]})-\frac{\varepsilon}{8}-\frac{\varepsilon}{4}-L(\gamma^1\mid_{[0,t_1]})\\
&\geq  d(\gamma^1({t_1}), \gamma^1({t_2})) - \frac{\varepsilon}{2}. \\
\end{aligned}
\end{equation}
Similarly, for each $n \in \mathbb{N}$, there exists a continuous curve $\gamma^n:[n-1, n] \to \mathcal{Y}$ connecting $y_{n-1}$ to $y_n$ (recall that $d(y_{n-1},y_n)=1$) such that
\begin{equation*}
u(\gamma^n({t_1})) - u(\gamma^n({t_2})) \geq d(\gamma^n(t_1),\gamma^n(t_2)) - \frac{\varepsilon}{2^n}
\end{equation*}
for all $n-1 \leq t_1 \leq t_2 \leq n$.
We can construct a continuous curve $\gamma:[0,\infty) \to \mathcal{Y}$ by
\begin{equation*}
\gamma(t) = \gamma^n(t) \text{ when } t \in [n-1, n) \ \text{for all} \ t \geq 0.
\end{equation*}

For any $0 \leq t_1 \leq t_2$, assume without loss of generality that $t_1 \in [n_1-1, n_1)$ and $t_2 \in [n_2-1, n_2)$. We have
\begin{equation*}
\begin{aligned}
&u(\gamma({t_1})) - u(\gamma({t_2})) \\
=& u(\gamma^{n_1}({t_1})) - u(\gamma^{n_1}({n_1}))+ \sum_{k=n_1+1}^{n_2-1} (u(\gamma^k({k-1}))- u(\gamma^k({k})))\\
& + u(\gamma^{n_2}({n_2-1})) - u(\gamma^{n_2}({t_2})) \\
\geq & d(\gamma^{n_1}(n_1),\gamma^{n_1}(t_1)) - \frac{\varepsilon}{2^{n_1}} + \sum_{k=n_1+1}^{n_2-1} \left(d(\gamma^k({k-1}), \gamma^k({k})) - \frac{\varepsilon}{2^{k}}\right) \\
&+ d(\gamma^{n_2}({n_2-1}),\gamma^{n_2}({t_2})) - \frac{\varepsilon}{2^{n_2}} \\
\geq & d(\gamma(t_2),\gamma(t_1))-\varepsilon.
\end{aligned}
\end{equation*}
The limit $\lim_{t \rightarrow \infty} d(\gamma(0),\gamma(t))=\infty$ follows from the 1-Lipschitz property of $u$ and $\lim_{t \rightarrow \infty}u(\gamma(t)) =-\infty$. Thus, the desired conclusion follows.
\end{proof}

\begin{Rem}
Lemma \ref{SSSSS}, Corollary \ref{SSSS} and Theorem \ref{CC} can be proven using the result of \cite{LQ}. For reader's convenience, we propose the proofs based on our discussion.
\end{Rem}

\begin{Pro}
Let $(\mathcal{Y},d)$ be a complete length space and $\Omega \subset \mathcal{Y}$ be an open set. Let $\mathcal{S}$ be a non-empty collection of metric viscosity solutions on $\Omega$, then the function $u(y)=\inf\{f(y):f \in \mathcal{S}\}$ (if it makes sense) is also a metric viscosity solution.
\end{Pro}

\begin{proof}
For any $x,y \in \Omega$ and $\varepsilon>0$, there exists $f \in \mathcal{S}$ such that $u(y)>f(y)-\varepsilon$.
According to Lemma \ref{local}, $f$ is 1-Lipschitz, we have
\begin{equation*}
u(x)-u(y)<f(x)-f(y)+\varepsilon \leq d(x,y)+\varepsilon.
\end{equation*}
By the arbitrariness of $\varepsilon$, it follows that $u$ is 1-Lipschitz.
On the other hand, for any $y \in \Omega$ and $x \in \Omega$ $(x \neq y)$, there exists $f_x \in \mathcal{S}$ such that
\begin{equation*}
f_x(y)<u(y)+d^2(x,y).
\end{equation*}
So we obtain
\begin{equation*}
\begin{aligned}
|\partial u|(y)=&\limsup_{x \rightarrow y}\frac{(u(y)-u(x))^+}{d(y,x)}\\
\geq &\limsup_{x \rightarrow y}\frac{(f_x(y)-d^2(x,y)-f_x(x))^+}{d(y,x)}\\
\geq & \limsup_{x \rightarrow y} \left[\frac{(f_x(y)-f_x(x))^+}{d(y,x)}-d(y,x) \right]\\
=&1.
\end{aligned}
\end{equation*}
By the 1-Lipschitz property of $u$ we have $|\partial u|(y) \leq1$, and thus $|\partial u|(y)=1$, it follows that $u$ is a metric viscosity solution.
\end{proof}

In \cite{LA1}, let $(\mathcal{Y},d)$ be a length space and $\Omega \subset \mathcal{Y}$ be an open set, the global slope of a function $u: \Omega \rightarrow \mathbb{R}$ is defined as
\begin{equation*}
l_u(y):=\sup_{x \in \Omega, x \neq y} \frac{(u(y)-u(x))^{+}}{d(y,x)}.
\end{equation*}
Generally speaking, if a continuous function $u$ is a metric viscosity solution on an open set $\Omega$, then $l_u(y)=1 \ \text{for all} \ y \in \Omega$. Conversely, this is not always true, as indicated by the following simple example.
\begin{ex}
Let
\[
u(x)=\begin{cases}
0 & \text{when} \ x \in (-\infty,0]\\
x & \text{when} \ x \in (0,\infty).
\end{cases}
\]
In this case, it is easily seen that $l_u(x)=1$ for all $x \in \mathbb{R}$ but it is not a metric viscosity solution on $\mathbb{R}$.
\end{ex}

However, if $\Omega$ is bounded, the following proposition holds.
\begin{Pro}
Let $(\mathcal{Y},d)$ be a complete length space and $\Omega \subset \mathcal{Y}$ be a bounded open set. A continuous function $u: \Omega \rightarrow \mathbb{R}$ is a metric viscosity solution if and only if for all $y \in \Omega$
\begin{equation}\label{slope1}
l_u(y)=1.
\end{equation}
\end{Pro}

\begin{proof}
We only need to prove the backward direction. It is not difficult to see that $u$ is 1-Lipschitz. Denote diam($\Omega$) as the diameter of $\Omega$. Suppose the contrary, there exist $y \in \Omega$, $\delta>0$ and $r>0$ such that
\begin{equation*}
\frac{(u(y)-u(x))^+}{d(y,x)} \leq 1-\delta \ \text{for all} \ x \in B_{2r}(y) \setminus \{y\}.
\end{equation*}
This implies
\begin{equation*}
u(y)-u(x) \leq d(y,x)-\delta r \ \text{for all} \ x \in \partial B_{r}(y).
\end{equation*}
By equality \eqref{slope1}, there exists $x^{*} \in \Omega \setminus B_{2r}(y)$ such that
\begin{equation*}
\frac{u(y)-u(x^*)}{d(y,x^*)} \geq 1-\frac{\delta r}{4K}
\end{equation*}
where $K=\max\{r,\text{diam}(\Omega)\}$. We have
\begin{equation}\label{slope3}
u(y)-u(x^*) \geq d(y,x^*)-d(y,x^*)\frac{\delta r}{4K} \geq d(y,x^*)-\frac{\delta r}{4}.
\end{equation}
Let $\gamma$ be a continuous curve connecting $y$ to $x^*$ such that $L(\gamma) \leq d(y,x^*)+\frac{\delta r}{2}.$ Let $y_r \in \gamma \cap \partial B_r(y)$, then
\begin{equation}\label{slope4}
\begin{aligned}
u(y)-u(x^*)&=u(y)-u(y_r)+u(y_r)-u(x^*)\\
&\leq d(y,y_r)-\delta r+d(y_r,x^*)\\
& \leq L(\gamma)-\delta r\\
& \leq d(y,x^*)-\frac{\delta r}{2}.
\end{aligned}
\end{equation}
Clearly, inequalities \eqref{slope3} and \eqref{slope4} are contradictory to each other.
\end{proof}

Next, we will discuss the stability of the metric viscosity solutions. Meanwhile, we will present an example that dl$_{C}$-function is not a metric viscosity solution.

On an open set in a complete Riemannian manifold $(\mathcal{M},g)$, viscosity solutions of the eikonal equation $|\triangledown u|_g=1$ are stable with respect to pointwise convergence. It is not difficult to get the same outcome on a general complete, locally compact length space $(\mathcal{Y},d)$. Specifically, assume that $\{u_{n}\}$ is a sequence of metric viscosity solutions on an open set $\Omega$ and $u$ is its pointwise limit. According to the local compactness of $\mathcal{Y}$, as in the proof of Lemma \ref{aaaa}, for any $y_{0} \in \Omega$ there exist $r>0$ and $y_{n} \in \partial B_{r}(y_{0})$ such that
\begin{equation*}
u_{n}(y_{0})-u_{n}(y_{n})=d(y_{0},y_{n})=r \ \text{for all} \ n \in \mathbb{N}.
\end{equation*}
Without loss of generality, assume that $y_{n}$ converges to $y_{1} \in \partial B_{r}(y_{0})$. We have
\begin{equation*}
\begin{aligned}
u(y_{0})-u(y_{1})=&u(y_{0})-u_{n}(y_{0})+u_{n}(y_{0})-u_{n}(y_{n})\\
&+u_{n}(y_{n})-u_{n}(y_{1})+u_{n}(y_{1})-u(y_{1}).
\end{aligned}
\end{equation*}
As $n \rightarrow \infty$ with $|u_{n}(y_{n})-u_{n}(y_{1})| \leq d(y_{n},y_{1})$, it follows that $u(y_{0})-u(y_{1})=r=d(y_{0},y_{1})$. Thus $u$ is a metric viscosity solution.

It is known that, in infinite-dimensional Banach spaces, such stability may fail for various types of viscosity solution. For details, we refer to \cite[Subsection 3.1]{MH}, \cite[Example 2.1]{AS} and \cite[Example 5.5]{YG}.
On the Wasserstein space, we can provide an example as follows.

First we denote
\begin{equation*}
\mathcal{D}=\left\{\sum_{i=1}^{m}\lambda_{i}\delta_{x_{i}} \bigg\vert \sum_{i=1}^{m}\lambda_{i}=1, \ \lambda_{i} \geq 0, \ x_{i} \in \mathcal{X} \ \text{for all} \ i=1,\dots, m, \ m \in \mathbb{N} \right\}.
\end{equation*}
It is well known that $\mathcal{D}$ is dense in $\mathcal{P}_{p}(\mathcal{X})$ ($p \geq 1$).

\begin{ex}\label{ex3}
Let $\omega_{n}=(1-\frac{1}{n^2})\delta_{0}+\frac{1}{n^2}\delta_{n^2} \in \mathcal{P}_{2}(\mathcal{\mathbb{R}})$ and $\Omega \subset \mathcal{P}_{2}(\mathcal{\mathbb{R}})\backslash \{\omega_n\}_{n \in \mathbb{N}}$ be an open set. Then, we observe that
$u_{n}(\omega)=W_{2}(\omega,\omega_{n})-W_{2}(\delta_{0},\omega_{n})$
is a (strong) metric viscosity solution on $\Omega$ for each $n \in \mathbb{N}$. However, the limit
$u(\omega)=\lim_{n \rightarrow \infty}[W_{2}(\omega,\omega_{n})-W_{2}(\delta_{0},\omega_{n})]$
is not a metric viscosity solution.

For any given $x_{i} \in \mathbb{R}, \lambda_{i}\in \mathbb{R}_{+}$, $i=1,\cdots,m \ \text{with} \ \sum_{i=1}^{m}\lambda_{i}=1$. Let $\pi_n$ be an optimal coupling between $\sum_{i=1}^{m}\lambda_{i}\delta_{x_{i}}$ and  $\omega_{n}$. Denote
$a_{i0}=\pi_n\{(x_{i},0)\}$ and
$a_{in^2}=\pi_n\{(x_{i},n^2)\},$
then $\sum_{i=1}^{m}a_{i0}=1-\frac{1}{n^2}, \ \sum_{i=1}^{m}a_{in^2}=\frac{1}{n^2}$. So we have
\begin{equation*}
\begin{aligned}
&W_{2}(\sum_{i=1}^{m}\lambda_{i}\delta_{x_{i}},\omega_{n})-W_{2}(\delta_{0},\omega_{n})\\
=&W_{2}(\sum_{i=1}^{m}\lambda_{i}\delta_{x_{i}}, \omega_{n})-n\\
=&\{\sum_{i=1}^{m}a_{i0}x_{i}^2+\sum_{i=1}^{m}a_{in^2}(x_{i}-n^2)^2\}^{\frac{1}{2}}-n\\
=&\{\sum_{i=1}^{m}a_{i0}x_{i}^2+n^2+\sum_{i=1}^{m}a_{in^2}x_{i}^2-2n^2\sum_{i=1}^{m}a_{in^2}x_{i}\}^{\frac{1}{2}}-n\\
=&\frac{\sum_{i=1}^{m}a_{i0}x_{i}^2+\sum_{i=1}^{m}a_{in^2}x_{i}^2-2n^2\sum_{i=1}^{m}a_{in^2}x_{i}}{\sqrt{\sum_{i=1}^{m}a_{i0}x_{i}^2+\sum_{i=1}^{m}a_{in^2}x_{i}^2-2n^2\sum_{i=1}^{m}a_{in^2}x_{i}+n^2}+n}.
\end{aligned}
\end{equation*}
Since $\sum a_{in^2}=\frac{1}{n^2}$ and
$\lim_{n \rightarrow \infty}2n\sum_{i=1}^{m}a_{in^2}x_{i}=0,$
\begin{equation*}
\lim_{n \rightarrow \infty }[W_{2}(\sum_{i=1}^{m}\lambda_{i}\delta_{x_{i}},\omega_{n})-W_{2}(\delta_{0},\omega_{n})]=0.
\end{equation*}
By the arbitrariness of $\sum_{i=1}^{m}\lambda_{i}\delta_{x_{i}}$, the above equality holds
on a dense subset of $\Omega$. We claim that
\begin{equation*}
u(\omega)=\lim_{n \rightarrow \infty}[W_2(\omega,\omega_n)-W_2(\delta_0,\omega_n)]=0 \ \text{on} \ \Omega.
\end{equation*}
Assume the contrary that there exists a subsequence $\{W_2(\cdot,\omega_{n_k})-W_2(\delta_0,\omega_{n_k})\}$ that does not converge to the 0-constant function. By the Ascoli-Arzel$\grave{\text{a}}$ Theorem \ref{A-A}, there exists a subsequence of $\{W_2(\cdot,\omega_{n_k})-W_2(\delta_0,\omega_{n_k})\}$ that converges to the 0-constant function, and thus we get a contradiction.
As a result, $u$ obviously is not a metric viscosity solution.
\end{ex}

Example \ref{ex3} also shows that a dl$_{C}$-function is not a metric viscosity solution in general, contrasting with the classical case.
However, the stability of the metric viscosity solutions with respect to locally uniform convergence holds, as stated in the following proposition.

\begin{Pro}[Stability of Metric Viscosity Solutions]
Let $(\mathcal{Y},d)$ be a complete length space and $\Omega \subset \mathcal{Y}$ be an open set. Let $\{u_{n}\}$ be a sequence of metric viscosity solutions on $\Omega$. If for each $y \in \Omega$ there exists $r>0$ with $B_r(x) \subset \Omega$ such that $u_n$ converges uniformly to $u$ on $B_r(x)$, then $u$ is also a metric viscosity solution on $\Omega$.
\end{Pro}
\begin{proof}
For $y \in \Omega$ and some $r>0$, suppose that $u_{n}$ converges uniformly to $u$ on the closed ball $\bar{B}_{r}(y)$. Then for any $\varepsilon>0$ and some positive integer $N$, we have
\begin{equation*}
|u_{N}(y')-u(y')|<\frac{\varepsilon}{3} \ \text{for any} \ y' \in \bar{B}_{r}(y).
\end{equation*}
As in the proof of Lemma \ref{aaaa}, there exists $y_{\varepsilon} \in \partial B_{r}(y)$ such that
\begin{equation*}
u_{N}(y)-u_{N}(y_{\varepsilon}) \geq r-\frac{\varepsilon}{3}.
\end{equation*}
Therefore,
\begin{equation*}
u(y)-u(y_{\varepsilon}) \geq r-\varepsilon.
\end{equation*}
By Lemma \ref{local}, the conclusion follows.
\end{proof}

\section{Existence of metric viscosity solutions on the Wasserstein space}
In this section, we will prove Theorems \ref{The1.2} and \ref{Pro111}, thus obtain the existence of the strong metric viscosity solutions on $\mathcal{P}_{p}(\mathcal{X})$ by two different ways. By Theorem \ref{CC}, it is not difficult to see that a 1-Lipschitz function $u: \mathcal{P}_{p}(\mathcal{X}) \rightarrow \mathbb{R}$ is a strong metric viscosity solution if and only if for any $\omega \in \mathcal{P}_{p}(\mathcal{X})$, there exists at least one negative gradient ray initiated from $\omega$.

Firstly, we aim to show the existence of co-rays for a sequence satisfying the (CS) condition in the Wasserstein space. It acts as a key element in the analysis of certain class of dl$_C$-functions (namely strong metric viscosity solutions). In \cite{HB}, Busemann introduced the concept of co-ray in G-spaces, which was extended to the Wasserstein space in \cite{ZG1}. Although traditionally defined in connection with a Busemann function, a co-ray fundamentally represents a limit of a sequence of geodesics tending to infinity. Herein, we will extend the concept.

\begin{defn}[Co-ray]
Let $(\mathcal{Y},d)$ be an unbounded geodesic space, and let $y_0 \in \mathcal{Y}$ and $\{y^{n}\}$ be a sequence in $\mathcal{Y}$.
A ray $\gamma$ is called a co-ray from $y_0$ to $\{y^{n}\}$ if there exist:
\begin{itemize}
\item{a subsequence $\{y^{n_{k}}\}$ with $d(y_0,y^{n_{k}}) \rightarrow \infty$;}
\item{a geodesic $\gamma^{n_{k}}$ connecting $y_0$ to $y^{n_{k}}$ for each $n_{k} \in \mathbb{N}$,}
\end{itemize}
such that $\gamma$ is the pointwise limit of $\{\gamma^{n_{k}}\}$.
\end{defn}

To facilitate the subsequent discussion, we first generalize Theorem \ref{The2.7} slightly to obtain the following lemma.
\begin{Lem}\label{Lem3.7}
Let $(\mathcal{Y},d)$ be a Polish space and $p \geq 1$. Let $\{\mu^n\}_{n \in \mathbb{N}} \subset \mathcal{P}_p(\mathcal{Y})$ and $\mu \in \mathcal{P}_p(\mathcal{Y})$, then the following statements are equivalent:
\begin{enumerate}
  \item[(i)]For some convergent sequence $\{\nu^{n}\} \subset \mathcal{P}_{p}(\mathcal{Y})$, {$\mu^{n} \Rightarrow \mu$ and
      $$\lim_{R \rightarrow \infty}\limsup_{n \rightarrow \infty}\int_{d(x,y) \geq R}d^{p}(x,y)d\pi^{n}=0,$$
      where $\pi^{n}$ is a coupling between $\mu^{n}$ and $\nu^{n}$;}
  \item[(ii)]{$W_{p}(\mu^{n},\mu) \rightarrow 0$ as $n \rightarrow \infty$.}
\end{enumerate}
\end{Lem}

This lemma follows from Theorem \ref{The2.7} and the triangle inequality,  and the proof is rather straightforward. For the sake of brevity, the details are omitted.

\begin{Lem}\label{Lem3.10}
On the Wasserstein space $\mathcal{P}_{p}(\mathcal{X})\ (p>1)$, for a sequence $\{\omega_n\}$, whether the (CS) condition holds is independent of the choices of the positive real number $\sigma$, base point $\omega_{0}$ and specific geodesics $\{\mu^{n}(t)\}_{0 \leq t \leq W_{p}(\omega_{0},\omega_{n})}$.
\end{Lem}

\begin{proof}
We only show that the (CS) condition is independent of the specific choice of $\{\mu^{n}(t)\}_{0 \leq t \leq W_{p}(\omega_{0},\omega_{n})}$. The proofs for other cases are analogous and omitted here.

Let $\{\nu^{n}(t)\}_{0 \leq t \leq L_{n}}$ and $\{\rho^{n}(t)\}_{0 \leq t \leq L_{n}}$ be any two unit-speed geodesics connecting $\omega_{0}$ to $\omega_{n}$. By Lemma \ref{Lem2.4}, for each $n$, there exists $\alpha^{n}, \ \beta^{n} \in \mathcal{P}(C[0,L_n], \mathcal{X})$ such that for any $0 \leq t \leq L_{n}$, we have
$\nu^{n}(t)=(e_{t})_{\sharp}\alpha^{n},\ \rho^{n}(t)=(e_{t})_{\sharp}\beta^{n}.$
Assume without loss of generality that $L_{n} \geq \sigma$ for all $n \in \mathbb{N}$. Let $\pi^{n}=(e_{0},e_{\sigma})_{\sharp}\alpha^{n} \in Opt(\nu^{n}(0),\nu^{n}(\sigma))$ and $\xi^{n}=(e_{0},e_{\sigma})_{\sharp}\beta^{n} \in Opt(\rho^n(0),\rho^{n}(\sigma))$. For $\alpha^{n}$-almost every $\gamma^{n} \in C([0,L_{n}],\mathcal{X})$ and $\beta^{n}$-almost every $\varsigma^{n} \in C([0,L_{n}],\mathcal{X})$, we have
\begin{equation*}
\begin{aligned}
\int_{d(x,y) \geq R}d^{p}(x,y)d\pi^{n}&=\int_{d(\gamma^{n}(0),\gamma^{n}(\sigma))\geq R}d^{p}(\gamma^{n}(0),\gamma^{n}(\sigma)) d \alpha^{n}(\gamma^{n})\\
&=(\frac{\sigma}{L_{n}})^{p}\int_{d(\gamma^{n}(0),\gamma^{n}(L_{n})) \geq \frac{L_{n}R}{\sigma}}d^{p}(\gamma^{n}(0),\gamma^{n}(L_{n}))d \alpha^{n}(\gamma^{n})\\
&=(\frac{\sigma}{L_{n}})^{p}\int_{d(x,y) \geq \frac{L_{n}R}{\sigma}}d^{p}(x,y)d(e_{0},e_{L_{n}})_{\sharp}\alpha^{n}\\
&=(\frac{\sigma}{L_{n}})^{p}\int_{d(x,y) \geq \frac{L_{n}R}{\sigma}}d^{p}(x,y)d(e_{0},e_{L_{n}})_{\sharp}\beta^{n}\\
&=(\frac{\sigma}{L_{n}})^{p}\int_{d(\varsigma^{n}(0),\varsigma^{n}(L_{n})) \geq \frac{L_{n}R}{\sigma}}d^{p}(\varsigma^{n}(0),\varsigma^{n}(L_{n}))d\beta^{n}(\varsigma^{n})\\
&=\int_{d(\varsigma^{n}(0),\varsigma^{n}(\sigma)) \geq R}d^{p}(\varsigma^{n}(0),\varsigma^{n}(\sigma))d\beta^{n}(\varsigma^{n})\\
&=\int_{d(x,y) \geq R}d^{p}(x,y)d\xi^{n},
\end{aligned}
\end{equation*}
where the second (resp. sixth) equality follows from that each $\gamma^{n}$ (resp. $\varsigma^{n}$) is a random geodesic. Taking $\lim_{R \rightarrow \infty}\limsup_{n \rightarrow \infty}$ on both sides yields
\begin{equation}\label{3.7}
\lim_{R \rightarrow \infty}\limsup_{n \rightarrow \infty} \int_{d(x,y) \geq R}d^{p}(x,y)d\pi^{n}=\lim_{R \rightarrow \infty}\limsup_{n \rightarrow \infty} \int_{d(x,y) \geq R}d^{p}(x,y)d\xi^{n}.
\end{equation}
According to Lemma \ref{CON}, for any $\sigma>0$ there exist subsequences of $\{\mu^n(\sigma)\}$ and $\{\nu^{n}(\sigma)\}$ converge weakly in $\mathcal{P}_p(\mathcal{X})$. By Lemma \ref{Lem3.7} and equality \eqref{3.7}, the conclusion follows.
\end{proof}

On a complete, separable, unbounded and locally compact geodesic space, the (CS) condition always holds when the sequence goes to infinity, guaranteed by the local compactness. Although $\mathcal{P}_{p}(\mathcal{X})$ lacks local compactness, we can still provide two examples that satisfy the condition:
\begin{ex}\label{Ex1}
The sequence $\{\omega_n\}$ diverges to infinity and lies on a ray in $\mathcal{P}_p(\mathcal{X})$.
\end{ex}

\begin{ex}\label{ex111}
The Dirac measures $\{\delta_{x_n}\}$ diverge to infinity in $ \mathcal{P}_p(\mathcal{X})$.
\end{ex}

Although there are many examples in $\mathcal{P}_p(\mathcal{X})$, as mentioned above, which satisfy the (CS) condition, we can also present an example where the condition does not hold as follows.

\begin{ex}\label{ex5}
Let $\omega_{n}=(1-\frac{1}{n^{p}})\delta_{0}+\frac{1}{n^{p}}\delta_{n^{2}} \in \mathcal{P}_{p}(\mathbb{R})$ $(p>1)$. Then the $p$-Wasserstein distance between $\omega_{n}$ and $\delta_{0}$ is given by
$$W_{p}(\omega_{n},\delta_{0})=\left(\frac{1}{n^{p}}|n^{2}-0|^{p} \right)^{\frac{1}{p}}=n,$$
which diverges to infinity as $n$ goes to infinity.
$\{(1-\frac{1}{n^{p}})\delta_{0}+\frac{1}{n^{p}}\delta_{tn} \}_{0 \leq t \leq n}$ is the unique unit-speed geodesic connecting $\delta_{0}$ to $\omega_{n}$. Now, consider set
$$\left\{ \left\{(1-\frac{1}{n^{p}})\delta_{0}+\frac{1}{n^{p}}\delta_{tn} \right\}_{0 \leq t \leq n}\right\}_{n \in \mathbb{N}} \cap \partial B_{1}(\delta_{0})=\left\{(1-\frac{1}{n^{p}})\delta_{0}+\frac{1}{n^{p}}\delta_{n} \right\}_{n \in \mathbb{N}}.$$
We claim that the sequence $\left\{(1-\frac{1}{n^{p}})\delta_{0}+\frac{1}{n^{p}}\delta_{n} \right\}_{n \in \mathbb{N}}$  has no convergent subsequence. Suppose, for contradiction, that there exist a subsequence (still denoted by $n$ for simplicity) and a probability measure $\mu \in \mathcal{P}_{p}(\mathbb{R})$ such that
$$\lim_{n \rightarrow \infty}W_{p}((1-\frac{1}{n^{p}})\delta_{0}+\frac{1}{n^{p}}\delta_{n}, \mu)=0.$$
By \cite[Proposition 5.1.8]{LA1}, we have $\mu=\delta_{0}$, which is a contradiction.
\end{ex}

\begin{Lem}\label{The1.1}
Assume $\mu, \mu^{n} \in \mathcal{P}_p(\mathcal{X})$ $(p>1)$ such that $\lim_{n \rightarrow \infty} W_{p}(\mu,\mu^{n})=\infty$, then there exists a co-ray from $\mu$ to $\{\mu^n\}$ if and only if $\{\mu^{n}\}$ satisfies the (CS) condition.
\end{Lem}
\begin{proof}
The forward direction follows from the definition of the (CS) condition. Thus, we only need to prove the backward direction. Let $\gamma^n$ be a unit-speed geodesic connecting $\omega$ and $\omega^n$, by Lemma \ref{Lem3.10}, we know that $\{\gamma^n(t)\}$ is relatively compact in $ \mathcal{P}_p(\mathcal{X})$ for all $t \geq 0$. By the Ascoli-Arzel$\grave{\text{a}}$ Theorem \ref{A-A}, we obtain the conclusion.
\end{proof}

Consider a sequence $\{H_n\}$ of non-empty closed subsets in $\mathcal{P}_{p}(\mathcal{X})$ such that, for any given $\omega_{0} \in \mathcal{P}_{p}(\mathcal{X})$, the distance $W_{p}(\omega_0,H_n)$ diverges to infinity as $n$ goes to infinity. We define
\begin{equation*}
u_{n}(\omega)=W_{p}(\omega,H_{n})-W_{p}(\omega_{0},H_{n}).
\end{equation*}
By the Ascoli-Arzel$\grave{\text{a}}$ Theorem \ref{A-A}, up to a subsequence, we suppose that
$$u(\omega)= \lim_{n \rightarrow \infty} u_{n}(\omega).$$
Recall that $\{H_{n}\}$ satisfies the (CS) condition if, for any $\omega \in \mathcal{P}_{p}(\mathcal{X})$, there exists $\mu^{n} \in H_{n}$ for all $n \in \mathbb{N}$ such that
\begin{equation*}
\lim_{n \rightarrow \infty}[W_{p}(\omega,H_{n})-W_{p}(\omega,\mu^{n})]=0
\end{equation*}
and $\{\mu^{n}\}$ itself satisfies the (CS) condition.

\begin{ex}
On $\mathcal{P}_{p}(\mathcal{X})$, Example \ref{ex5} shows that $H_{n}=\{(1-\frac{1}{n^{p}})\delta_{0}+\frac{1}{n^{p}}\delta_{n^2}\}$ does not satisfy the (CS) condition. Based on Example \ref{ex111}, it is not difficult to see that $K_{n}=\{\delta_{x}: x \in \mathcal{X}, W_{p}(\delta_{x},\omega_{0})=n\}$ for any $\omega_{0} \in \mathcal{P}_{p}(\mathcal{X})$, do satisfy the (CS) condition.
\end{ex}

\begin{proof}[Proof of Theorem \ref{The1.2}]
By the assumption, for a given $\omega \in \mathcal{P}_{p}(\mathcal{X})$, there exists a sequence $\{\mu^{n}\}$, $\mu^{n} \in H_{n}$, satisfying the (CS) condition and
\begin{equation}\label{ff}
\lim_{n \rightarrow \infty} [W_{p}(\omega,H_{n})-W_{p}(\omega,\mu^{n})]=0.
\end{equation}
Let $\gamma_{n}:[0,l_{n}] \rightarrow \mathcal{P}_{p}(\mathcal{X})$ be a unit-speed geodesic connecting $\omega$ to $\mu^{n}$. By Lemma \ref{The1.1}, up to a subsequence, there exists a unit-speed ray $\gamma$ initiated from $\omega$ such that
\begin{equation*}
\gamma(t)=\lim_{n \rightarrow \infty}\gamma_{n}(t).
\end{equation*}

We claim that $\gamma$ is a negative gradient ray of $u$.
By equality \eqref{ff}, it follows that
\begin{equation*}
\begin{aligned}
&\limsup_{n \rightarrow \infty}[W_{p}(\gamma_{n}(t),\mu^{n})-W_{p}(\gamma_{n}(t),H_{n})]\\
=&\limsup_{n \rightarrow \infty}[W_{p}(\omega,\mu^{n})-W_{p}(\gamma_{n}(t),\omega)-W_{p}(\gamma_{n}(t),H_{n})]\\
\leq & \lim_{n \rightarrow \infty}[W_{p}(\omega,\mu^{n})-W_{p}(\omega,H_{n})]\\
=&0.
\end{aligned}
\end{equation*}
Combining with $W_{p}(\gamma_{n}(t),\mu^{n})-W_{p}(\gamma_{n}(t),H_{n}) \geq 0$, we obtain
\begin{equation}\label{gg}
\lim_{n \rightarrow \infty}[W_{p}(\gamma_{n}(t),\mu^{n})-W_{p}(\gamma_{n}(t),H_{n})]=0.
\end{equation}
Without loss of generality, we assume $0 \leq t_1 \leq t_{2} \leq l_{n}$ for all $n$. By equality \eqref{gg}, we have
$$\begin{aligned}
&u(\gamma(t_{1}))-u(\gamma(t_{2}))\\
=&\lim_{n \rightarrow \infty}[W_{p}(\gamma_{n}(t_{1}),H_{n})-W_{p}(\gamma_{n}(t_{2}),H_{n})] \\
= &\lim_{n \rightarrow \infty}\left[
W_{p}(\gamma_{n}(t_{1}),\mu^{n})-W_{p}(\gamma_{n}(t_{2}),\mu^{n}) \right]\\
=&\lim_{n \rightarrow \infty}\left[(l_{n}-t_{1})-(l_{n}-t_{2}) \right]\\
=&t_{2}-t_{1}\\
=&W_p(\gamma(t_2),\gamma(t_1)),
\end{aligned}$$
which concludes the argument.
\end{proof}

\begin{Lem}\label{Lem4.5}
Let $u$ be a 1-Lipschitz function on $\mathcal{P}_p(\mathcal{X})$, then any unit-speed curve $\gamma: [0, \infty) \rightarrow \mathcal{P}_{p}(\mathcal{X})$ satisfying $u(\gamma(t_1))-u(\gamma(t_2))=t_2-t_1$ for all $t_2 \geq t_1 \geq 0$ must be a ray.
\end{Lem}

\begin{proof}
For $[t_{1},t_{2}] \subset [0,\infty)$, by the 1-Lipschitz property of $u$, we have
$$W_{p}(\gamma(t_{1}),\gamma(t_{2})) \geq u(\gamma(t_{1}))-u(\gamma(t_{2}))=t_{2}-t_{1}=L(\gamma_{\mid_{[t_{1},t_{2}]}}) \geq W_{p}(\gamma(t_{1}),\gamma(t_{2})).$$
\end{proof}

To prove Theorem \ref{Pro111}, first we recall that the following set
\begin{equation*}
\mathcal{D}=\left\{\sum_{i=1}^{m}\lambda_{i}\delta_{x_{i}} \bigg\vert \sum_{i=1}^{m}\lambda_{i}=1, \ \lambda_{i} \geq 0, \ x_{i} \in \mathcal{X} \ \text{for all} \ i=1,\dots, m, \ m \in \mathbb{N} \right\}.
\end{equation*}
is dense in $\mathcal{P}_{p}(\mathcal{X})$.

\begin{proof}[Proof of Theorem \ref{Pro111}]
Let $\omega \in \mathcal{P}_{p}(\mathcal{X})$. By the 1-Lipschitz property of $u$ and the H$\ddot{\text{o}}$lder inequality, we obtain the following estimate:
\begin{equation}
|\hat{u}(\omega)|
\leq \int_{\mathcal{X}}d(x,x_{0})d\omega+|u(x_{0})|
\leq \left(\int_{\mathcal{X}}d^{p}(x,x_{0})d\omega \right)^{\frac{1}{p}}+|u(x_{0})|
\end{equation}
for any $x_{0} \in \mathcal{X}$. Thus, $\hat{u}$ is well-defined. Let $\omega_1, \omega_2 \in \mathcal{P}_p(\mathcal{X}).$
Applying the Kantorovich-Rubinstein formula {\cite[Theorem 5.16]{CV}} and the H$\ddot{\text{o}}$lder inequality, we have
$$\begin{aligned}
\hat{u}(\omega_{1})-\hat{u}(\omega_{2}) &=\int_{\mathcal{X}}u(x)d\omega_{1}-\int_{\mathcal{X}}u(x)d\omega_{2}\\
& \leq \sup_{f \ 1-\text{Lipschitz on} \ \mathcal{X}}  \left\{ \int_{\mathcal{X}}f(x)d\omega_{1}-\int_{\mathcal{X}}f(x)d\omega_{2} \right\}\\
& = W_{1}(\omega_{1},\omega_{2})\\
& \leq W_{p}(\omega_{1},\omega_{2}).
\end{aligned}$$
This implies that $\hat{u}$ is 1-Lipschitz on $\mathcal{P}_{p}(\mathcal{X})$.

Now consider $\omega=\sum_{i=1}^{m}\lambda_{i}\delta_{x_{i}} \in \mathcal{D}$, and let $\gamma^{i}:[0,\infty) \rightarrow \mathcal{X}$ be a unit-speed negative gradient ray of $u$ with $\gamma^{i}({0})=x_{i}$. Let
$$\mu({t})=\sum_{i=1}^{m}\lambda_{i}\delta_{\gamma^{i}({t})} \in \mathcal{P}_{p}(\mathcal{X}),$$
which is a unit-speed curve initiated from $\omega$ in $\mathcal{P}_{p}(\mathcal{X})$.
Given any $t_{1},t_{2} \in [0,\infty)$, it follows that
$$\begin{aligned}
\hat{u}(\mu(t_{1}))-\hat{u}(\mu(t_{2})) &=\int_{\mathcal{X}}u(x)d\mu(t_{1})-\int_{\mathcal{X}}u(x)d\mu(t_{2})\\
&= \sum_{i}^{m}\lambda_{i}u(\gamma^{i}(t_{1}))-\sum_{i}^{m}\lambda_{i}u(\gamma^{i}(t_{2}))\\
&= \sum_{i}^{m}\lambda_{i} \left(u(\gamma^{i}(t_{1}))- u(\gamma^{i}(t_{2}))\right)\\
&=\sum_{i}^{m}\lambda_{i}(t_{2}-t_{1})\\
&=t_{2}-t_{1},
\end{aligned}$$
where the fourth equality follows from the fact that each $\gamma^{i}$ is a unit-speed negative gradient ray of $u$.
By Lemma \ref{Lem4.5}, $\mu(t)$ is a unit-speed negative gradient ray of $\hat{u}$ initiated from $\omega$.

Up to now, we have proven that for each $\omega \in \mathcal{D}$, there exists at least one negative gradient ray of $\hat{u}$ initiated from $\omega$. Next, we will show that this fact holds on the entire space $\mathcal{P}_p(\mathcal{X})$. For any $\omega \in \mathcal{P}_p(\mathcal{X})$, there exists a sequence
\begin{equation*}
\omega_{m}=\sum_{i=1}^{m_k}\lambda_{m,i}\delta_{x_{m,i}} \in \mathcal{D} \ \ \text{such that} \ \ \lim_{m \rightarrow \infty}W_{p}(\omega_{m}, \omega)=0.
\end{equation*}
For each $i=1,\cdots,m_k$ and $m \in \mathbb{N}$, let $\gamma_{m,i}:[0,\infty) \rightarrow \mathcal{X}$ be a unit-speed negative gradient ray of $u$ with $\gamma_{m,i}(0)=x_{m,i}$. Let
\begin{equation*}
\mu_{m}(t)=\sum_{i=1}^{m_k}\lambda_{m,i}\delta_{\gamma_{m,i}(t)},
\end{equation*}
which is exactly a unit-speed negative gradient ray of $\hat{u}$ with $\mu_{m}(0)=\omega_{m}$.
For any $t \in [0, \infty)$, let
$\pi_{m,t}=\sum_{i=1}^{m_k}\lambda_{i}\delta_{(x_{m,i},\gamma_{m,i}(t))} \in \Pi(\omega_{m},\mu_{m}(t)).$
For $R>t$ it follows that
\begin{equation*}
\limsup_{m \rightarrow \infty}\int_{d(x,y) \geq R}d^{p}(x,y)\pi_{m,t}(dx,dy)=t^{p} \limsup_{m \rightarrow \infty}\int_{d(x,y) \geq R}\pi_{m,t}(dx,dy)=0.
\end{equation*}
Thus,
\begin{equation*}
\lim_{R \rightarrow \infty}\limsup_{m \rightarrow \infty}\int_{d(x,y) \geq R}d^{p}(x,y)\pi_{m,t}(dx,dy)=0.
\end{equation*}
According to Lemmas \ref{CON} and \ref{Lem3.7}, $\{\mu_{m}(t)\}$ is relatively compact in $\mathcal{P}_{p}(\mathcal{X})$ for all $t \in [0, \infty)$. By the Ascoli-Arzel$\grave{\text{a}}$ Theorem \ref{A-A}, up to a subsequence, there exists a unit-speed ray $\mu$ initiated from $\omega$ such that
\begin{equation*}
\mu(t)=\lim_{m \rightarrow \infty}\mu_m(t).
\end{equation*}
We have
\begin{equation*}
\hat{u}(\mu(t_{1}))-\hat{u}(\mu(t_{2}))
= \lim_{m \rightarrow \infty}[\hat{u}(\mu_{m}(t_{1}))-\hat{u}(\mu_{m}(t_{2}))]
=t_{2}-t_{1}
=W_p(\mu(t_2),\mu(t_1)).
\end{equation*}
As a result, the theorem is proven.
\end{proof}

The geodesic space $\mathcal{X}$ is called non-branching if any geodesic $\gamma: [a,b] \rightarrow \mathcal{X}$ is uniquely determined by its restriction to a nontrivial time-interval. For $p>1$, $\mathcal{X}$ is non-branching if and only if $\mathcal{P}_p(\mathcal{X})$ is non-branching (see \cite[Corollary 7.32]{CV}). Based on this fact, we have the following proposition.

\begin{Pro}[{\cite[Weak KAM Property]{FA}}]
Assume $\mathcal{X}$ is non-branching. Let $u$ be a strong metric viscosity solution on $\mathcal{P}_p(\mathcal{X}) \ (p>1)$. If $(\gamma(t))_{t \geq 0}$ is a unit-speed negative gradient ray of $u$, then for any $\tau>0$, the subray $(\gamma(t+\tau))_{t \geq 0}$ is the unique unit-speed negative gradient ray of $u$ initiated from $\gamma(\tau)$.
\end{Pro}

\begin{proof}
For $\tau>0$, assume $(\tilde{\gamma}(t))_{t \geq 0}$ is a unit-speed negative gradient ray of $u$ initiated from $\gamma(\tau)$. Let
\begin{equation*}
\gamma^{*}(t)=\left\{
\begin{aligned}
&\gamma(t), \  0 \leq t \leq \tau,\\
&\tilde{\gamma}(t-\tau), \ t \geq \tau.
\end{aligned}
\right.
\end{equation*}
It is easily seen that $u(\gamma^{*}(t_1))-u(\gamma^{*}(t_2))=t_2-t_1$ for all $t_2 \geq t_1 \geq 0$. By Lemma \ref{Lem4.5}, we know that $(\gamma^{*}(t))_{t\geq 0}$ is a ray. Since $\mathcal{P}_p(\mathcal{X})$ is non-branching, the ray $(\gamma^{*}(t))_{t\geq 0}$ coincides with $(\gamma(t))_{t\geq 0}$.
\end{proof}

Busemann functions on $\mathcal{P}_{p}(\mathcal{X})$, introduced in \cite{ZG1}, form a subclass of dl$_{C}$-functions. We will present a representation formula for any strong metric viscosity solutions by the Busemann functions. By Theorem \ref{Pro111}, we know that for any $\omega \in \mathcal{P}_p(\mathcal{X})$ $(p \geq 1)$ there exists at least one ray starting from $\omega$. For convenience, we recall the formal definition of the Busemann function on $\mathcal{P}_{p}(\mathcal{X})$ here.
\begin{defn}\label{Def4.7}
Let $\gamma: [0,\infty)\rightarrow \mathcal{P}_{p}(\mathcal{X})$ be any unit-speed ray starting from $\gamma(0)$, and the function $b_{\gamma}: \ \mathcal{P}_{p}(\mathcal{X}) \rightarrow \mathbb{R}$ defined as
$$b_{\gamma}(\omega):= \lim_{t \rightarrow \infty}[W_{p}(\omega,\gamma(t))-t]$$
is called the Busemann function associated to $\gamma$.
\end{defn}

On a complete, locally compact and non-compact geodesic space, any dl-functions can be represented by the Busemann functions \cite{Liang}. We will generalize the conclusion to the strong metric viscosity solutions on $\mathcal{P}_{p}(\mathcal{X})$ in the following proposition.

\begin{Pro}\label{Pro4.8} Let $u$ be a strong metric viscosity solution on $\mathcal{P}_p(\mathcal{X})$ and denote the set of all unit-speed negative gradient rays of $u$ by $\mathcal{N}(u)$, then
for any $\omega \in \mathcal{P}_{p}(\mathcal{X})$
\begin{equation*}
u(\omega)=\inf_{\gamma \in \mathcal{N}(u)}[u(\gamma(0))+b_{\gamma}(\omega)].
\end{equation*}
\end{Pro}

\begin{proof}
By equality \eqref{DL},
\begin{equation*}
u(\omega)=\lim_{n \rightarrow \infty}\{W_p(\omega,u^{-1}(-\infty,-n])-n\} \ \text{for all} \ \omega \in \mathcal{P}_p(\mathcal{X}).
\end{equation*}
Hence, for any $\gamma \in \mathcal{N}(u)$ and $t \geq 0$,
\begin{equation}\label{Busemann}
\begin{aligned}
u(\omega)&=\lim_{n \rightarrow \infty}\{W_p(\omega,u^{-1}(-\infty,-n])-n\}\\
& \leq [W_p(\omega,\gamma(t))-t]+\lim_{n \rightarrow \infty}\{W_p(\gamma(t),u^{-1}(-\infty,-n])-n+t\}\\
&=[W_p(\omega,\gamma(t))-t]+u(\gamma(t))+t.
\end{aligned}
\end{equation}
Since $\gamma \in \mathcal{N}(u)$, we have
\begin{equation*}
u(\gamma(t))+t=u(\gamma(0)).
\end{equation*}
Thus the inequality \eqref{Busemann} reads as $u(\omega) \leq [W_p(\omega,\gamma(t))-t]+u(\gamma(0)).$ Sending $t \rightarrow +\infty$, we obtain
\begin{equation}
u(\omega) \leq \inf_{\gamma \in \mathcal{N}(u)}[u(\gamma(0))+b_{\gamma}(\omega)].
\end{equation}
On the other hand, by the property of the strong metric viscosity solution, there exists a unit-speed negative gradient ray $\gamma_{\omega}:[0,\infty) \rightarrow \mathcal{P}_{p}(\mathcal{X})$ starting from $\omega \in \mathcal{P}_{p}(\mathcal{X})$. Thus
\begin{equation*}
u(\omega)=u(\gamma_{\omega}(0))+b_{\gamma_{\omega}}(\omega) \geq  \inf_{\gamma \in \mathcal{N}(u)}[u(\gamma(0))+b_{\gamma}(\omega)].
\end{equation*}
This completes the proof of the proposition.
\end{proof}

\end{document}